\documentclass[11pt]{article}
\usepackage{graphicx} 
\usepackage{geometry}

\usepackage{amsmath}
\usepackage{amssymb}
\usepackage{amsthm} 
\usepackage{dsfont}
\usepackage{xcolor}
\newtheorem{theorem}{Theorem}[section] 
\newtheorem{lemma}[theorem]{Lemma} 
\newtheorem{definition}[theorem]{Definition}
\newtheorem{proposition}[theorem]{Proposition}
\newtheorem{corollary}[theorem]{Corollary}
\theoremstyle{remark}
\newtheorem{remark}[theorem]{Remark}
\newtheorem{example}[theorem]{Example}
\usepackage{hyperref}
\usepackage{authblk}
\usepackage{tikz}
\usetikzlibrary{shapes.geometric,angles,quotes}

\newcommand{\sech}{\mathrm{sech}}

\title{Inverse problems and continuity of heat content coefficients}

\author[1]{Ryoki Endo \thanks{\texttt{endo@m.sc.niigata-u.ac.jp}}}

\author[2]{Katie Gittins \thanks{\texttt{katie.gittins@durham.ac.uk}}}

\author[3]{Jacob Walker
\thanks{\texttt{jpw79@cam.ac.uk}}}

\affil[1]{\footnotesize Faculty of Science, Niigata University, 8050 Ikarashi 2-no-cho, Nishi-ku, Niigata 950-2181, Japan.}

\affil[2]{\footnotesize Department of Mathematical Sciences, Durham University, Mathematical Sciences and Computer Science Building, Upper Mountjoy Campus, Stockton Road, Durham University, DH1 3LE, United Kingdom.}

\affil[3]{\footnotesize Trinity College, Cambridge, CB2 1TQ, United Kingdom.}

\date{\today}

\begin{document}
\maketitle

\begin{abstract}
    We consider the small-time asymptotic expansions for the Dirichlet and Open heat contents of convex planar domains.
    We prove that the small-time asymptotic expansion for the Open heat content can determine a triangle among triangles, a parallelogram among parallelograms, and an acute trapezium among acute trapezia of the same height.
    In addition, for a sequence of convex polygons converging to a convex set with smooth boundary in the sense of Hausdorff, we show that the coefficients of order $t$ in the small-time asymptotic expansions of the Dirichlet and the Open heat contents converge to those of the limiting set.
    While, for a sequence of convex sets with smooth boundary converging to a convex polygon in the sense of Hausdorff, we show that the coefficients of order $t$ do not converge to that of the limiting polygon.
\end{abstract}

\section{Introduction}

\subsection{The Dirichlet and Open Heat Contents}
We consider heat flow problems on bounded, convex, open sets in $\mathbb{R}^2$.
We impose initial temperature conditions and boundary conditions (e.g. Dirichlet) on such a set and then let the heat flow out of it. Our quantity of interest, the \emph{heat content} is the amount of heat left inside the set at a later time. We focus on two heat flow problems that give rise to the Dirichlet heat content and the Open heat content. For further properties of the heat content in the Euclidean setting, see for example \cite{vdB13, vdBGG20, vdBG15, vdBG25, FG23}.

We first define the Dirichlet heat content. Let $D \subset \mathbb{R}^2$ be a bounded, convex, open set in $\mathbb{R}^2$.
Let $\tilde{u}_D : D \times [0, \infty) \to \mathbb{R}$ be the unique solution of
\begin{equation*}
    \begin{cases}
        \Delta \tilde{u}(x,t) = \frac{\partial \tilde{u}(x,t)}{\partial t}, & x\in D, t>0,\\
        \lim_{t\downarrow 0}\tilde{u}(x,t) = 1, & x \in D,\\
        \lim_{x \to x_0}\tilde{u}(x,t) = 0, & x_0 \in \partial D, t>0.
    \end{cases}
\end{equation*}
Then the Dirichlet heat content of $D$ is defined as 
\begin{equation*}
    Q_D(t) := \int_D \tilde{u}_D(x,t) \, dx.
\end{equation*}

It was shown in \cite[Eq. (1.8)]{vdBGa94} that if $\Omega \subset \mathbb{R}^2$ is a bounded, convex, open set with $C^\infty$ boundary, then, as $t \downarrow 0$,
\begin{equation}\label{eq:Dir-smooth}
    Q_\Omega(t) = a_{0,Q}(\Omega) + a_{1/2,Q}(\Omega) t^{1/2} + a_{1,Q}(\Omega) t + O(t^{3/2}),
\end{equation}
with coefficients
\begin{equation}\label{eq:Dir-smooth-coeffs}
    a_{0,Q}(\Omega) = |\Omega|, \quad
    a_{1/2,Q}(\Omega) = - \frac{2L(\partial \Omega)}{\pi^{1/2}}, \quad
    a_{1,Q}(\Omega) = \pi,
\end{equation}
where $|\Omega|$ is the area of $\Omega$ and $L(\partial \Omega)$ is the length of the boundary of $\Omega$. 

Let $P \subset \mathbb{R}^2$ be a bounded, convex, open set with polygonal boundary that has interior angles $\gamma_1, \dots, \gamma_N$ where $\gamma_i \in (0,\pi)$.
It was shown in \cite[Theorem 1]{vdBS90} that, as $t \downarrow 0$,
\begin{equation}\label{eq:Dir-poly}
    Q_P(t) = a_{0,Q}(P) + a_{1/2,Q}(P) t^{1/2} + a_{1,Q}(P) t + O(e^{-q/t}),
\end{equation}
with coefficients
\begin{equation}\label{eq:Dir-poly-coeffs}
    a_{0,Q}(P) = |P|, \quad
    a_{1/2,Q}(P) = - \frac{2L(\partial P)}{\pi^{1/2}}, \quad
    a_{1,Q}(P) = \sum_{i=1}^N c(\gamma_i) ,
\end{equation}
where $c: (0, \pi] \to \mathbb{R}$ is defined as
\begin{equation}\label{eq:Dir-angle}
    c(\alpha) = \int_0^\infty \frac{4\sinh((\pi - \alpha)\theta)}{\sinh(\pi\theta)\cosh(\alpha\theta)}d\theta,
\end{equation}
and $q>0$ is a positive constant that depends on $P$.
We observe that $c(\pi) = 0$.
We note that analogous statements to \eqref{eq:Dir-smooth} and \eqref{eq:Dir-poly} are shown to hold in \cite{vdBGa94}, \cite{vdBS90} respectively without assuming the sets are convex. 

Next, we define the Open heat content. Let $D \subset \mathbb{R}^2$ be a bounded, convex, open set in $\mathbb{R}^2$.
Let $u_D : \mathbb{R}^2 \times [0, \infty) \to \mathbb{R}$ be the unique solution of
\begin{equation*}
    \begin{cases}
        \Delta u(x,t) = \frac{\partial u(x,t)}{\partial t}, & x\in \mathbb{R}^2, t>0,\\
        \lim_{t \downarrow 0} u_D(x,t) = \mathds{1}_D(x), & x \in \mathbb{R}^2 \setminus \partial D.
    \end{cases}
\end{equation*}
We note that this heat flow problem takes place in all of $\mathbb{R}^2$ and that no boundary condition has been imposed.
Then the Open heat content of $D$ is defined as 
\begin{equation*}
    H_D(t) := \int_D u_D(x,t) \, dx.
\end{equation*}

From \cite[Theorem 1.3 and Corollary 1.7]{vdBGy15}, we see that if $\Omega \subset \mathbb{R}^2$ is a bounded, convex, open set with smooth boundary, then, as $t \downarrow 0$, $H_\Omega(t)$ has a small-time asymptotic expansion in the same form as \eqref{eq:Dir-smooth}
with coefficients
\begin{equation}\label{eq:Open-smooth-coefficients}
    a_{0,H}(\Omega) = |\Omega|, \quad 
    a_{1/2,H}(\Omega) = - \frac{L(\partial \Omega)}{\pi^{1/2}}, \quad 
    a_{1,H}(\Omega) = 0.
\end{equation}

Let $P \subset \mathbb{R}^2$ be a bounded, convex, open set with polygonal boundary that has interior angles $\gamma_1, \dots, \gamma_N$ where $\gamma_i \in (0,\pi)$.
It was shown in \cite[Theorem 1]{vdBG16} that, as $t \downarrow 0$, $H_P(t)$ has a small-time asymptotic expansion in the same form as \eqref{eq:Dir-poly} with coefficients
\begin{equation}\label{eq:Open-poly-coeffs}
    a_{0,H}(P) = |P|, \quad 
    a_{1/2,H}(P) = - \frac{L(\partial P)}{\pi^{1/2}}, \quad
    a_{1,H}(P) = \sum_{i=1}^N g(\gamma_i),
\end{equation}
where $g: (0, \pi] \to \mathbb{R}$ is defined as
\begin{equation}\label{eq:Open-angle}
    g(\alpha) = \frac{1}{\pi} + \left(1- \frac{\alpha}{\pi}\right)\cot(\alpha).
\end{equation}
We observe that $g$ is non-negative and that $g(\pi)=0$ via l'Hôpital's rule.
We note that analogous small-time asymptotic expansions are shown to hold in \cite{vdBGy15}, \cite{vdBG16} without assuming the sets are convex.   

\subsection{Inverse problems for heat content}
Recall that the Dirichlet heat trace also enjoys analogous small-time asymptotic expansions in these settings.
Let $\Omega \subset \mathbb{R}^2$ be a bounded, convex, open set with smooth boundary and Dirichlet eigenvalues $\{\lambda_k(\Omega)\}_{k \in \mathbb{N}}$. Then, as $t \downarrow 0$,
\begin{equation}\label{eq:Trace-smooth}
    \sum_{k =1}^\infty e^{-t \lambda_k(\Omega)} = \frac{a_{-1,T}(\Omega)}{t} + \frac{a_{-1/2,T}(\Omega)}{t^{1/2}} + a_{0,T}(\Omega) + O(t^{1/2}),
\end{equation}
with coefficients
\begin{equation}\label{eq:Trace-smooth-coeffs}
    a_{-1,T}(\Omega) = \frac{|\Omega|}{4\pi}, \quad
    a_{-1/2,T}(\Omega) = - \frac{L(\partial \Omega)}{8\pi^{1/2}}, \quad 
    a_{0,T}(\Omega) = \frac{1}{6}.
\end{equation}
See, for example, \cite{MR25} and references therein. 
Let $P \subset \mathbb{R}^2$ be a bounded, convex, open set with polygonal boundary that has interior angles $\gamma_1, \dots, \gamma_N$ where $\gamma_i \in (0,\pi)$ and Dirichlet eigenvalues $\{\lambda_k\}_{k \in \mathbb{N}}$.
It was shown in \cite[Theorem 1]{vdBS88} that, as $t \downarrow 0$,
\begin{equation}\label{eq:Trace-poly}
    \sum_{k =1}^\infty e^{-t \lambda_k(P)} = \frac{a_{-1,T}(P)}{t} + \frac{a_{-1/2,T}(P)}{t^{1/2}} + a_{0,T}(P) + O(e^{-s/t}),
\end{equation}
with coefficients
\begin{equation}\label{eq:Trace-poly-coeffs}
    a_{-1,T}(P) = \frac{|P|}{4\pi}, \quad 
    a_{-1/2,T}(P) = - \frac{L(\partial P)}{8\pi^{1/2}}, \quad
    a_{0,T}(P) = \sum_{i=1}^N h(\gamma_i),
\end{equation}
where $h: (0, \pi] \to \mathbb{R}$ is defined as
\begin{equation}\label{eq:Trace-angle}
    h(\alpha) = \frac{\pi^2 - \alpha^2}{24 \pi \alpha}
\end{equation}
and $s>0$ is a positive constant that depends on $P$.

A classic question in spectral geometry is whether we can hear the shape of a drum, that is, whether knowing the eigenvalues of the Dirichlet Laplacian allows us to determine the underlying shape.
The small-time asymptotics of the heat trace give us a link between the Dirichlet eigenvalues of the set and some geometric quantities of the set and are a key tool in such inverse problems.
The question of whether the Dirichlet spectrum determines the heat content was posed in \cite{Gil09}, a construction for manifolds that are not isospectral but have the same heat content for all time was also provided.
Moreover, several examples of planar domains that are isospectral (either Dirichlet or mixed Dirichlet and Neumann) but do not have the same heat content for all time were constructed in \cite{vdBDK14}. 

Directly related to the inverse results of this paper, in \cite{GM13}, it was shown that the first three coefficients in the small-time asymptotic expansion of the heat trace uniquely determine a triangle within the collection of triangles.
Moreover, in \cite{LR15}, by employing the small-time asymptotic expansion of the heat trace, it was shown that if two parallelograms have the same Dirichlet spectrum, then they are congruent, and if two acute trapezia have the same Dirichlet spectrum, then they are congruent (see Definition \ref{def:acute} for the definition of an acute trapezium).

Regarding the Dirichlet heat content, in \cite{MM17}, it was shown that the first three coefficients in the small-time asymptotic expansion of the Dirichlet heat content uniquely determine a triangle within the collection of triangles.
Moreover, in \cite{B18}, it was shown that the first three coefficients in the small-time asymptotic expansion of the Dirichlet heat content uniquely determine parallelograms within the collection of parallelograms, and uniquely determine acute trapezia within the collection of acute trapezia. 

In this paper, we investigate analogous inverse problems for the Open heat content. Our result for triangles is as follows.

\begin{theorem}\label{thm:triangles}
    The first three coefficients in the small-time asymptotic expansion of the Open heat content uniquely determine a triangle within the collection of triangles.
\end{theorem}
\noindent
In addition, we have the corresponding result for parallelograms.
\begin{theorem}\label{thm:parallel}
    The first three coefficients in the small-time asymptotic expansion of the Open heat content uniquely determine a parallelogram within the collection of parallelograms.
\end{theorem}
\noindent
We recall the following definition of an acute trapezium from \cite[p. 828]{LR15}.

\begin{definition}\label{def:acute}
    A trapezium is called acute if it is a convex, open, set in $\mathbb{R}^2$ with four sides such that the two parallel sides have lengths $b, B$ where $B > b$ is the base, and the other two sides have lengths $\ell, \ell'$. The angles at the base are denoted $\alpha$ and $\beta$, and must satisfy
    $$\alpha + \beta < \frac{\pi}{2}.$$
\end{definition}
\noindent
For acute trapezia, we have the following.

\begin{theorem}\label{thmtrapezium}
    If two acute trapezia have the same height $h$ and the same first three terms in the small-time asymptotic expansion for the Open heat content, then they are congruent.
\end{theorem}

\subsection{Continuity of heat content coefficients}
Whether the coefficients in the small-time asymptotic formula for the Dirichlet heat trace are continuous along certain sequences of convex sets was investigated in \cite{MR25}.    
More precisely, it was shown in \cite[Theorem 6.4]{MR25} that if $P \subset \mathbb{R}^2$ is a bounded, convex polygon, and if $(\Omega_k)_k \subset \mathbb{R}^2$ is a sequence of bounded, convex, open sets with smooth boundary such that $\Omega_k$ converges to $P$ in the sense of Hausdorff, then
\begin{equation*}
    \lim_{k \to \infty} a_{-1,T}(\Omega_k) = a_{-1,T}(P) \mbox{ and }
    \lim_{k \to \infty} a_{-1/2,T}(\Omega_k) = a_{-1/2,T}(P),
\end{equation*}
but
\begin{equation*}
        \lim_{k \to \infty} a_{0,T}(\Omega_k) \neq a_{0,T}(P).
\end{equation*}
On the other hand, it is shown in \cite[Theorem 6.5]{MR25}, that if $\Omega \subset \mathbb{R}^2$ is a bounded, convex, open set with smooth boundary, and if $(P_k)_k \subset \mathbb{R}^2$ is a sequence of bounded, convex polygons such that $P_k$ converges to $\Omega$ in the sense of Hausdorff, then the first three heat trace coefficients of $P_k$ converge to those of $\Omega$. 

Inspired by \cite{MR25}, we investigate the continuity of the third coefficients $a_{1,Q}$ and $a_{1,H}$ in the small-time asymptotic expansions of the Dirichlet and Open heat contents respectively.

For a regular $n$-gon in $\mathbb{R}^2$, the term of order $t$ in the small-time asymptotic expansion of the Open heat content is
\begin{equation*}
    \frac{n}{\pi} + 2\cot\left(\left(\frac{n-2}{n}\right) \pi \right) = O \left(\frac{1}{n}\right), \quad n \to \infty.
\end{equation*}
(see \cite[p. 2234]{vdBG16}). In the limit, this coefficient agrees with the corresponding one for the disk.
For the general case of approximating a convex set with smooth boundary via convex polygons, we have the following.
\begin{theorem}\label{thm:conv}
    Let $\Omega \subset \mathbb{R}^2$ be a bounded, convex, open set with smooth boundary. Suppose that $(P_k)_k \subset \mathbb{R}^2$ is a sequence of bounded, convex, open sets, each with polygonal boundary, such that $P_k$ converges to $\Omega$ in the sense of Hausdorff, denoted $P_k \overset{H}{\to} \Omega$, as $k \to \infty$. Then, 
    \begin{equation*}
        \lim_{k \to \infty} a_{1,Q}(P_k) = a_{1,Q}(\Omega)
        \mbox{ and }
        \lim_{k \to \infty} a_{1,H}(P_k) = a_{1,H}(\Omega).
\end{equation*}
\end{theorem}

On the other hand, for the approximation of a convex polygon via convex sets with smooth boundary, we have the following.
\begin{theorem}\label{thm:nonconv}
    Let $P \subset \mathbb{R}^2$ be a bounded, convex, open set with polygonal boundary. Suppose that $(\Omega_k)_k \subset \mathbb{R}^2$ is a sequence of bounded, convex, open sets with smooth boundary such that $\Omega_k$ converges to $P$ in the sense of Hausdorff, denoted $\Omega_k \overset{H}{\to} P$, as $k \to \infty$. Then, 
\begin{equation*}
        \lim_{k \to \infty} a_{1,Q}(\Omega_k) \neq a_{1,Q}(P),
        \mbox{ and }
        \lim_{k \to \infty} a_{1,H}(\Omega_k) \neq a_{1,H}(P).
\end{equation*}
\end{theorem}

\subsection*{Plan of the paper}
In Section \ref{sec:inverse}, we prove the inverse results for the Open heat content, Theorems \ref{thm:triangles}, \ref{thm:parallel}, and \ref{thmtrapezium}.
In Section \ref{sec:continuity}, we prove Theorems \ref{thm:conv} and \ref{thm:nonconv}.
A central factor in the arguments of all proofs is that the term of order $t$ in the small-time asymptotic expansion, defined by \eqref{eq:Open-angle}, is a convex function.

\subsection*{Acknowledgements}
Part of this work was undertaken while RE was a visiting researcher at Durham University supported by JSPS KAKENHI Grant Numbers JP24KJ1170, JP26K22276.
Part of this work was carried out while JW was undertaking a summer research project in summer 2026 at Durham University supported by a London Mathematical Society Undergraduate Research Bursary and the Department of Mathematical Sciences, Durham University.
The authors are grateful to Yulun Wu who undertook a summer research project at Durham University in summer 2023 where preliminary work on this project was carried out.

\section{Inverse problems for the Open heat content}
\label{sec:inverse}

\subsection{Triangles}\label{sec:triangles}
In this section, we show that the small-time asymptotic expansion of the Open heat content can determine a triangle in the collection of triangles. We follow the strategy that was developed in \cite{GM13} for the heat trace.

As in \cite{GM13}, we consider a triangle $T \subset \mathbb{R}^2$ with interior angles $\alpha, \beta, \gamma$ and we recall that
\begin{equation}\label{eq:hclink}
    \frac{L(\partial T)^2}{4|T|} = \cot{\frac{\alpha}{2}} + \cot{\frac{\beta}{2}} + \cot{\frac{\gamma}{2}}.
\end{equation}
Note that the left-hand side is determined by the first two terms in the small-time asymptotic expansion of the Open heat content \eqref{eq:Open-poly-coeffs}.

In addition, we observe that the third term in the small-time asymptotic expansion of the Open heat content \eqref{eq:Open-poly-coeffs} which is defined by $g$ in \eqref{eq:Open-angle} is convex on $(0, \pi]$. Indeed, on $(0,\pi)$, we have that $g$ is twice differentiable, and taking derivatives we have
    \begin{align*}
    g'(\alpha) &= \frac{-1}{2\pi\sin^2(\alpha)}(2(\pi-\alpha) + \sin(2 \alpha))<0, \\
    g''(\alpha) &= \frac{2}{\pi \sin^2(\alpha)}(1 + (\pi - \alpha)\cot(\alpha)) = \frac{2 g(\alpha)}{\sin^2(\alpha)},
    \end{align*}
and $g(\pi)=0$.

We define $w: (0,\pi) \to \mathbb{R}$ as
\begin{equation*}
    w(\theta) := \cot{\frac{\theta}{2}},
\end{equation*}
and observe that, for $\theta \in (0,\pi)$,
\begin{align}
    w'(\theta) &= -\frac{1}{2(\sin (\theta/2))^2} < 0, \label{eq:sdec}\\
    w''(\theta) &= \frac{\cos(\theta/2)}{2(\sin(\theta/2))^3} >0. \label{eq:sconv}
\end{align}
So $w$ is strictly decreasing and strictly convex on $(0,\pi)$.

Analogously to \cite[Proposition 1]{GM13}, we prove the following.
\begin{proposition}\label{prop:Tangles}
    A triple $(\alpha, \beta, \gamma)$ of positive real numbers satisfying $\alpha + \beta + \gamma = \pi$ is uniquely determined, up to ordering, by the values of
    \begin{align}\label{eq:apfnct}
       W(\alpha, \beta, \gamma) := w(\alpha) + w(\beta) + w(\gamma)
       = \cot{\frac{\alpha}{2}} + \cot{\frac{\beta}{2}} + \cot{\frac{\gamma}{2}},
    \end{align}
    and
    \begin{align}
        G(\alpha, \beta, \gamma) 
        &:= g(\alpha) + g(\beta) + g(\gamma) \nonumber\\
        &= \frac{3}{\pi} + \left(1- \frac{\alpha}{\pi}\right)\cot(\alpha) + \left(1- \frac{\beta}{\pi}\right)\cot(\beta) + \left(1- \frac{\gamma}{\pi}\right)\cot(\gamma). \label{eq:anglefnct}
    \end{align}
\end{proposition}

As in \cite{GM13}, the proof of Theorem \ref{thm:triangles} then follows from Proposition \ref{prop:Tangles}.

\begin{proof}[Proof of Theorem \ref{thm:triangles}]
    From the small-time asymptotic expansion of $H_T(t)$ \eqref{eq:Open-poly-coeffs}, we know $|T|$, $L(\partial T)$ and $G(\alpha, \beta, \gamma)$ and so, from \eqref{eq:hclink} and Proposition \ref{prop:Tangles}, we can determine the triangle up to dilation. The area of the triangle allows us to determine the dilation factor and hence the triangle.
\end{proof}

So it remains to prove Proposition \ref{prop:Tangles}. Let
\begin{equation*}
    \mathcal{D} := \{(\alpha, \beta, \gamma) : \alpha, \beta, \gamma >0, \alpha + \beta + \gamma = \pi\} \subset (0,\pi)\times (0,\pi)\times (0,\pi).
\end{equation*}
Then $\mathcal{D}$ is the space of all possible configurations of angles of a triangle. The union of the lines that bisect each of the interior angles of $\mathcal{D}$ corresponds to the collection of all isosceles triangles. The equilateral triangle corresponds to the intersection of these three lines.

To prove Proposition \ref{prop:Tangles}, we require the following analogue of \cite[Lemma 1]{GM13}.
\begin{lemma}[c.f. Lemma 1 from \cite{GM13}]\label{lem:GM}
    Consider the functions $W$, $G$, and $V(\alpha, \beta, \gamma) = \alpha + \beta +\gamma$ on $(0,\pi)\times (0,\pi)\times (0,\pi)$. Then
    \begin{enumerate}
        \item[(a)] The function $G$ is strictly convex on $(0,\pi)\times (0,\pi)\times (0,\pi)$.
        \item[(b)] The gradients $\nabla G$, $\nabla W$, $\nabla V$ are linearly independent at all non-isosceles points of $\mathcal{D}$.
    \end{enumerate}
\end{lemma}

With Lemma \ref{lem:GM} in hand, the proof of Proposition \ref{prop:Tangles} follows via the same strategy as that developed in \cite{GM13} which we recall here for the reader's convenience. The aim is to show that $W$ is monotone on the level sets of $G$.

\begin{proof}[Proof of Proposition \ref{prop:Tangles}]
The strict convexity of $G$ implies that the sublevel set 
$$ \mathcal{G}_{\leq s} = \{p \in (0,\pi)\times (0,\pi)\times (0,\pi) : G(p) \leq s\}$$
is strictly convex for any $s>0$. Note that
$$ \mathcal{G}_{s} = \{p \in (0,\pi)\times (0,\pi)\times (0,\pi) : G(p) = s\}$$
is the boundary of $\mathcal{G}_{\leq s}$.
It follows that the sets $\mathcal{G}_s \cap \mathcal{D}$ are either empty, the point $(\frac{\pi}{3}, \frac{\pi}{3}, \frac{\pi}{3})$, or a closed curve in the interior of $\mathcal{D}$ containing the point $(\frac{\pi}{3}, \frac{\pi}{3}, \frac{\pi}{3})$.
The point $(\frac{\pi}{3}, \frac{\pi}{3}, \frac{\pi}{3})$ is determined by $G$ alone and $G(\frac{\pi}{3}, \frac{\pi}{3}, \frac{\pi}{3})= (3/\pi) + 2 \cot(\pi/3)$. 
For $s > (3/\pi) + 2 \cot(\pi/3)$, we need to show that $W$ is strictly monotone along the part of $\mathcal{G}_s \cap \mathcal{D}$ between the two endpoints $p, q$ that correspond to isosceles triangles in the region where $\alpha \leq \beta \leq \gamma$.
If $W$ is not strictly monotone along this curve, then there would be a point $r$ lying on this curve between $p$ and $q$ where $W$ is stationary. By the Lagrange Multiplier Theorem, $\nabla W$ would then be a linear combination of $\nabla G$ and $\nabla V$ contradicting Part \textit{(b)} of Lemma \ref{lem:GM}. 
\end{proof}

\begin{proof} [Proof of Lemma \ref{lem:GM}]
    For part \textit{(a)}, the matrix of second derivatives of $G$ is a diagonal matrix with entries $g''(\alpha)$, $g''(\beta)$, $g''(\gamma)$.
    Since $g$ is strictly convex on $(0,\pi)$, the matrix of second derivatives of $G$ is positive definite for all $(\alpha, \beta, \gamma) \in (0,\pi)\times (0,\pi)\times (0,\pi)$. Hence $G$ is is strictly convex.

    For part \textit{(b)}, we suppose that there is a non-isosceles point $(\alpha, \beta, \gamma)$ (that is, $\alpha, \beta, \gamma$ are distinct) and $R, S, T \in \mathbb{R}$ not all zero such that
    \begin{equation*}
        R\nabla G + S\nabla W + T \nabla V = (0,0,0),
    \end{equation*}
    that is
    \begin{equation}\label{eq:zeros}
        Rg'(\theta) + Sw'(\theta) + T = 0,
    \end{equation}
    has three distinct solutions $\theta = \alpha$, $\theta= \beta$, $\theta= \gamma$.
    Our goal is to show that \eqref{eq:zeros} has at most two solutions.

    We observe that $R$ and $S$ cannot both be zero.
    If $R=0$, $S, T \neq 0$, then $w'(\theta) = -T/S$. But $w$ is strictly convex on $(0,\pi)$, so $w'$ is strictly increasing and $w'(\theta) = T/S$ can have at most one solution. We can treat the case where $S=0$, $R, T \neq 0$ similarly using that $g$ is strictly convex on $(0,\pi)$.
    
    It remains to treat the case where $R, S \neq 0$.
    Set 
    \begin{equation}\label{eq:K}
        K(\theta) = Rg'(\theta) + Sw'(\theta) + T.
    \end{equation}
    Then $K'(\theta) = Rg''(\theta) + Sw''(\theta)$.
    Since $g, w$ are strictly convex on $(0,\pi)$, if $R,S$ have the same sign, then $K$ is monotone and so has at most one zero.
    
    The case where $R,S$ have different signs is more delicate as $K$ need not always be monotone or convex for all values of $R, S$. 
    In order for $K$ to have three distinct zeros, it would need to have two turning points by Rolle's Theorem. 
    To show that $K$ has at most two zeros, it is sufficient to show that $K$ has at most one turning point.

    We define $M:(0,\pi) \to \mathbb{R}$ as
    \begin{equation}\label{eq:monotone}
        M(\theta) = \frac{g''(\theta)}{w''(\theta)}.
    \end{equation}
    Then
    \begin{equation*}
        M(\theta) = \frac{1}{2\pi (\cos(\theta/2))^4}(\sin \theta + (\pi - \theta)\cos \theta),
    \end{equation*}
    and
    \begin{equation*}
        M'(\theta) = \frac{(\sin(\theta/2))^2}{\pi (\cos(\theta/2))^5} (2\cos(\theta/2) - (\pi - \theta)\sin(\theta/2)).
    \end{equation*}
    Now we define $L: (0,\pi) \to \mathbb{R}$ as
    \begin{equation*}
        L(\theta) := 2\cos(\theta/2) - (\pi - \theta)\sin(\theta/2),
    \end{equation*}
    and observe that
    \begin{equation*}
        L'(\theta) = - \frac{(\pi - \theta)}{2} \cos(\theta/2) < 0
    \end{equation*}
    for $\theta \in (0,\pi)$. So $L$ is decreasing on $(0,\pi)$ and $L(\pi) =0$. Hence $L(\theta) > 0$ on $(0,\pi)$ and we conclude that $M$ is increasing on $(0,\pi)$.

    We return to $K'(\theta)$ and write
    \begin{equation*}
        K'(\theta) = Sw''(\theta)\left(1 + \frac{R}{S}M(\theta)\right).
    \end{equation*}
    For $K$ to have a stationary point at $\theta$, we must have $K'(\theta)=0$ which is equivalent to $M(\theta) = -S/R$. Since $M$ is increasing, there can only be at most one solution and so $K$ has at most one turning point.
\end{proof}

\subsection{Quadrilaterals}\label{subsec:quad}
In this section, we first show that the small-time asymptotic expansion of the Open heat content can determine a parallelogram $P \subset \mathbb{R}^2$ within the collection of parallelograms. We follow the strategy that was developed in \cite{LR15} for the heat trace.

\begin{proof}[Proof of Theorem \ref{thm:parallel}] 
    Let $P \subset \mathbb{R}^2$ be a parallelogram with interior angles $\alpha, \alpha, \pi-\alpha, \pi-\alpha$ where $\alpha \in (0, \frac{\pi}{2}]$.
    It was shown in \cite[p. 827]{LR15} that the area $|P|$, the perimeter $L(\partial P)$, and the interior angle $\alpha \in (0, \frac{\pi}{2}]$ determine the parallelogram $P$.
    
    From the small-time asymptotic expansion of $H_P(t)$ \eqref{eq:Open-poly-coeffs}, we know $|P|$ and $L(\partial P)$.
    So, it remains to show that we can determine $\alpha$ from the term of order $t$ in the small-time asymptotic expansion for the Open heat content which is
    \begin{equation*}
         \frac{4}{\pi} + 2 \left(1-\frac{2\alpha}{\pi} \right) \cot \alpha := p(\alpha).
    \end{equation*}
    For $\alpha \in (0, \frac{\pi}{2})$, consider
    \begin{align*}
        p'(\alpha) &= -\frac{4}{\pi} \cot \alpha - \frac{2}{(\sin \alpha)^2} \left(1-\frac{2\alpha}{\pi} \right) < 0. 
    \end{align*}
    We see that $p$ is decreasing on $(0, \frac{\pi}{2})$ and $p(\frac{\pi}{2})=\frac{4}{\pi}$.
    Thus we can uniquely determine the angle $\alpha$ from the term linear in $t$, concluding the proof. 
\end{proof}

\begin{remark}
    In particular, from the proof of Theorem \ref{thm:parallel}, since $p$ is decreasing on $(0, \frac{\pi}{2})$, we see that the small-time asymptotic expansion of the Open heat content can distinguish a rectangle from a parallelogram with $\alpha \in (0, \frac{\pi}{2})$.
\end{remark}

Next we prove Theorem \ref{thmtrapezium} for acute trapezia by following a similar strategy to that used in \cite{GM13} and in Section \ref{sec:triangles} via Lagrange multipliers.

\begin{proof}[Proof of Theorem \ref{thmtrapezium}]
    Let $Z$ be an acute trapezium with interior angles $\alpha, \beta, \pi - \alpha, \pi - \beta$ such that $\alpha + \beta < \frac{\pi}{2}$. We have that
    $$ |Z| = \frac{(B+b)}{2}h \Rightarrow B + b = \frac{2|Z|}{h},$$
    and
    $$ L(\partial Z) = B + b + \ell + \ell'.$$
    Since we are given the height $h$, the first term of the small-time asymptotic expansion of the Open heat content \eqref{eq:Open-poly-coeffs} uniquely determines $B+b$, and hence the second term also uniquely determines $\ell + \ell'$.

    We have that $\ell = h \csc \alpha$ and $ \ell' = h \csc \beta$ so
    $$ \csc \alpha + \csc \beta = \frac{\ell +\ell'}{h}.$$
    By \eqref{eq:Open-poly-coeffs}, the term of order $t$ in the small-time asymptotic expansion for the Open heat content is
    $$ \frac{4}{\pi} + \left(1 - \frac{2\alpha}{\pi}\right) \cot \alpha + \left(1 - \frac{2\beta}{\pi}\right) \cot \beta.$$
    We wish to prove that the following system of equations has a unique solution:
    \begin{align*}
      \csc \alpha + \csc \beta & = c_1, \\
       \left(1 - \frac{2\alpha}{\pi}\right) \cot \alpha + \left(1 - \frac{2\beta}{\pi}\right) \cot \beta & = c_2.
    \end{align*}
    Consider the triangle
    $$T := \left\{(\alpha, \beta) \in \mathbb{R}^2 : \alpha, \beta >0, \alpha+\beta < \frac{\pi}{2}\right\}$$
    and the triangle $\tilde{T} \subset T$ defined as
    $$\tilde{T} := \left\{(\alpha, \beta) \in \mathbb{R}^2 : \alpha, \beta >0, \alpha+\beta < \frac{\pi}{2}, \beta \leq \alpha\right\}.$$
    Note that $T$ is the space of all configurations of the angles of an acute trapezium, and $\tilde{T}$ is the quotient space of such configurations up to swapping the roles of $\alpha, \beta$.
    We set 
    \begin{equation*}
        Q(\alpha, \beta) := \csc \alpha + \csc \beta,
    \end{equation*}
    and
    \begin{equation*}
        G(\alpha, \beta) := \left(1 - \frac{2\alpha}{\pi}\right) \cot \alpha + \left(1 - \frac{2\beta}{\pi}\right) \cot \beta.
    \end{equation*}
    We aim to show that $Q$ is monotone on the level sets of $G$ in $\tilde{T}$.
    The matrix of second derivatives of $G$ is a diagonal matrix with entries 
    \begin{equation*}
        \frac{\partial ^2G}{\partial \alpha^2}(\alpha ,\beta)
        = \frac{1}{\pi (\sin \alpha)^4}
        (2(\sin \alpha)^2(1-\cos 2\alpha) + \sin 2\alpha (\sin 2\alpha + (\pi-2\alpha))) > 0
    \end{equation*}
    for $\alpha \in (0, \frac{\pi}{2})$, and similarly for the partial derivative with respect to $\beta$. So the Hessian of $G$ is positive definite. Hence $G$ is strictly convex and, as a consequence, the sub-level sets of $G$
    $$ \mathcal{G}_{\leq c} = \{p \in T : G(p) \leq c\}$$
    are convex. 
    The level sets $$ \mathcal{G}_c = \{p \in T : G(p) = c\}$$ cannot be convex subsets contained in $\tilde{T}$ otherwise they would not be convex sets in $T$ as they would be disconnected. Similarly, the level sets cannot intersect the line $\alpha = \frac{\pi}{2} - \beta$ more than once.
    In addition, 
    $$\lim_{\alpha \to 0} G(\alpha, \beta) = \infty, \quad \lim_{\beta \to 0} G(\alpha, \beta) = \infty$$
    so the sub-level sets of $G$ do not touch the edges $\alpha = 0$, $\beta = 0$ of $T$.
    We also have that the function $\tilde{G}(\alpha) := G(\alpha, \alpha)$ satisfies
    $$ \tilde{G}'(\alpha) = -\frac{(2\sin 2\alpha + 2(\pi - 2\alpha))}{\pi (\sin \alpha)^2}<0$$
    for $\alpha \in (0,\frac{\pi}{2})$. So $\alpha \mapsto \tilde{G}(\alpha)$ is monotone decreasing for $\alpha \in (0,\frac{\pi}{2})$ and the level set
    $$ \mathcal{G}_c= \{p \in T : G(p) = c\}$$
    can only intersect the line $\alpha = \beta$ at most once. 
    So the only possibility is for $\mathcal{G}_c\cap \tilde{T}$ to be a curve joining the line $\alpha = \beta$ to the line $\alpha = \frac{\pi}{2} - \beta$.
    
    For $\lambda \in \mathbb{R}$, we consider
    $$F(\alpha, \beta)= \lambda Q(\alpha,\beta) + G(\alpha,\beta).$$
    Assume that $F$ has an extremal point in $\mathcal{G}_c \cap \tilde{T}$. 
    Then, by the method of Lagrange multipliers, we must have that
    $$ \nabla F(\alpha, \beta) = (0,0).$$
    That the partial derivative with respect to $\alpha$ is equal to zero gives that
    \begin{equation}\label{eq:lmt}
        -\lambda \pi \cos \alpha = \sin 2\alpha + (\pi -2\alpha).
    \end{equation} 
    We note that for
    $$ r(\theta) = \frac{\sin 2\theta + (\pi -2\theta)}{\cos \theta},$$
    we have
    $$ r'(\theta) = \frac{\sin \theta (\pi - 2\theta - \sin (2\theta))}{(\cos\theta)^2} >0$$
    for $\theta \in (0,\frac{\pi}{2})$.
    So \eqref{eq:lmt} can have at most one solution.
    As the extremal point lies in $\mathcal{G}_c \cap \tilde{T}$, we would require two distinct solutions.
    We conclude that $Q$ is monotone on $\mathcal{G}_c\cap \tilde{T}$ as required.
\end{proof}

\begin{remark}
    We remark that in the case of Dirichlet boundary conditions, the Dirichlet spectrum determines the height via the wave trace and its relationship with the shortest closed geodesic (see \cite{LR15} and references therein). As we consider the Open heat content, we do not have additional spectral information from which we can glean the height.
\end{remark}

The following example shows that there exists a pair of acute trapezia with different heights that have the same first three terms in the small-time asymptotic expansion for the Open heat content. So it is necessary to prescribe the height to be able to distinguish such trapezia from one another.

\begin{example}
Let $Z$ be an acute trapezium with base angles $\alpha, \beta$. Let $x:=\cot \alpha + \cot \beta$ and $y := \csc \alpha + \csc \beta$. Note that $x$ and $y$ are just constants. By elementary geometry, we have that
    \begin{align}
        2b+hx+hy &= L(\partial Z), \label{eq:eqn1} \\
        h(2b+hx) &= 2|Z|. \label{eq:eqn2}
    \end{align}
    From \eqref{eq:eqn1} we get $2b = L(\partial Z)-hx-hy$, and substituting that into \eqref{eq:eqn2}, we get
    \[h(L(\partial Z)-hy)=2|Z|.\]
    This is a quadratic equation in $h$ with solutions
    \[h=\frac{L(\partial Z) \pm \sqrt{L(\partial Z)^2-8|Z|y}}{2y}.\] 
    Since we require that $b > 0$, some combinations of $(|Z|,L(\partial Z))$ will only give us one $h$ that is valid. However, for some other combinations, both $h$ will be valid and they lead to two non-congruent acute trapezoids. 
    For example, consider an acute trapezium with base angles $\alpha = \frac{\pi}{6} = \beta$,
    base of length $4\sqrt{3}+1$ where the other side-lengths are $4, 1, 4$ and the height is $h=2$. This trapezium has the same first three terms in the small-time Open heat content asymptotic expansion as a trapezium with the same base angles, base of length $\frac{\sqrt{3}}{2}+7$ where the other side-lengths are $1+2\sqrt{3}, 1-\frac{\sqrt{3}}{2}, 1+2\sqrt{3}$ and the height is $h = \frac{1}{2}(1+2\sqrt{3})>2$.
\end{example}

\begin{remark}
    The rectangle is distinguished within the collection of all trapezia.
    Consider a trapezium with internal angles $0<\beta \leq \alpha < \pi$, $\pi-\alpha$ and $\pi-\beta$, then by \eqref{eq:Open-poly-coeffs}, the term of order $t$ in the small-time asymptotic expansion of Open heat content is
    \begin{equation*}
        \frac{4}{\pi} + \left(1 - \frac{2\alpha}{\pi}\right) \cot \alpha + \left(1 - \frac{2\beta}{\pi}\right) \cot \beta := \frac{4}{\pi} + f(\alpha) + f(\beta) := F(\alpha, \beta),
    \end{equation*}
    where $f(x) = \left(1 - \frac{2x}{\pi}\right) \cot x$.
    We observe that $f(\pi -x) = f(x)$ and $f$ is strictly monotone decreasing for $x \in \left( 0, \frac{\pi}{2} \right)$, hence $f$ is strictly monotone increasing for $x \in \left(\frac{\pi}{2}, \pi \right)$ . We also have that $f(x) = 0 \iff x = \frac{\pi}{2}$. Hence, we conclude that $F(\alpha, \beta) \geq \frac{4}{\pi}$, and that equality only holds when $\alpha = \beta = \frac{\pi}{2}$.
    The first two terms of \eqref{eq:Open-poly-coeffs} give us the area and perimeter from which we can uniquely determine the rectangle.
\end{remark}

Note that it is possible to find a parallelogram and a trapezium that have the same first three terms in the small-time asymptotic expansion of the Open heat content. 
For example, we could consider the trapezium with $B=3$, $b=1$, and $\alpha = \beta = \frac{\pi}{4}$ and the parallelogram with $\alpha = \frac{\pi}{4}$, $L=2$, $W = \sqrt{2}$, see Figure \ref{fig:trappar}.

\begin{figure}[h]
    \centering
    \begin{tikzpicture}[my angle/.style={font=\scriptsize, draw, angle eccentricity=2, angle radius=3mm}]
  \node (a) [trapezium, trapezium angle=45, minimum width=50mm, draw, thick, label=above:1, label=below:3, label=right:$\sqrt{2}$, label=left:$\sqrt{2}$] {};
  \draw [densely dashed] (a.north west) coordinate (a nw) -- (a nw |- a.south) node [midway,right] {$1$} coordinate (a1) (a.north east) coordinate (a ne) -- (a ne |- a.south) node [midway,left] {$1$} coordinate (a2);
  \draw (a nw |- a.south) ++(0,1.5mm) -| ++(-1.5mm,-1.5mm) (a ne |- a.south) ++(0,1.5mm) -| ++(1.5mm,-1.5mm);
  \coordinate (a blc) at (a.bottom left corner);
  \coordinate (a brc) at (a.bottom right corner);
  \pic [my angle, "$\frac{\pi}{4}$"] {angle=a1--a blc--a nw};
  \pic [my angle, "$\frac{\pi}{4}$"] {angle=a ne--a brc--a1};
  \pic [my angle, "$\frac{\pi}{4}$"] {angle=a blc--a nw--a1};
  \pic [my angle, "$\frac{\pi}{4}$"] {angle=a2--a ne--a brc};
    \end{tikzpicture}
    \hspace{2cm}
    \begin{tikzpicture}[my angle/.style={font=\scriptsize, draw, angle eccentricity=2, angle radius=3mm}]
      \node (a) [trapezium, trapezium left angle=45, trapezium right angle=135, minimum width=50mm,
      draw, thick, label=below:$2$, label=left:$\sqrt{2}$, label=above:$2$, label=right:$\sqrt{2}$] {};
        \draw [densely dashed] (a.north west) coordinate (a nw) -- (a nw |- a.south) node [midway,right] {$1$} coordinate (a1);
        \draw (a nw |- a.south) ++(0,1.5mm) -| ++(-1.5mm,-1.5mm);
          \coordinate (a blc) at (a.bottom left corner);
          \coordinate (a brc) at (a.bottom right corner);
           \pic [my angle, "$\frac{\pi}{4}$"] {angle=a1--a blc--a nw};       
           \pic [my angle, "$\frac{\pi}{4}$"] {angle=a blc--a nw--a1};
    \end{tikzpicture}
    \caption{A trapezium and a parallelogram that have the same first three terms of the small-time asymptotic expansion for the Open heat content.}
    \label{fig:trappar}
\end{figure}
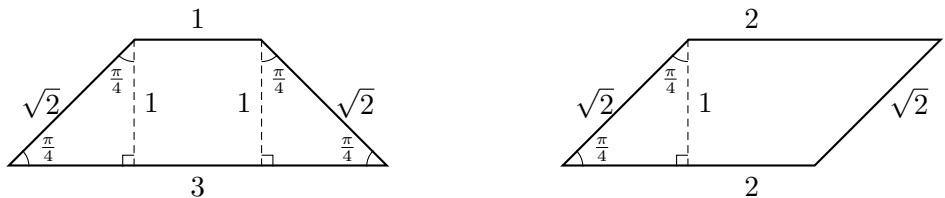

\section{The continuity of heat content coefficients}\label{sec:continuity}

We first give a general statement related to the $a_{1,Q}(\cdot)$ and $a_{1,H}(\cdot)$ coefficients (i.e., the coefficients of $t$ in the small-time asymptotic expansions for the heat contents). Consider $\Omega \subset \mathbb{R}^2$, a bounded, convex, open set with polygonal boundary, and a quantity of the form
$$a_1(\Omega) = \sum_{i=1}^N f(\gamma_i)$$
where $f\colon (0, \pi] \to \mathbb{R}$ is a function and $A = \{\gamma_1, \dots, \gamma_N\}$ is the set of interior angles of $\Omega$ with each $\gamma_i \in (0,\pi)$. Define
$$\mathcal{C} := \{f \colon (0,\pi] \to \mathbb{R} \colon f(\pi) = 0, f \ \text{is differentiable and convex}\}.$$

\begin{theorem}\label{thrm:general-bounds}
    Let $\Omega$ be a bounded, convex, open set with polygonal boundary and let $f \in \mathcal{C}$. Let $\beta := \min A$. Then 
        $$a_1(\Omega) \leq \frac{2\pi f(\beta)}{\pi - \beta}.$$
    Moreover,
        $$a_1(\Omega) \geq -2\pi f'(\pi),$$
    where equality holds if and only if $f$ is affine on $[\beta, \pi]$.
    In particular, if $f$ is strictly convex near $\pi$, then $a_1(\Omega) > -2\pi f'(\pi)$.
\end{theorem}

Before we prove Theorem \ref{thrm:general-bounds}, we make the following observation about the upper bound it provides.

\begin{lemma}
    Let $f \in \mathcal{C}$. Then the function 
    $$h(\gamma) :=  
    \begin{cases}
        \frac{2\pi f(\gamma)}{\pi-\gamma}, &\gamma \in (0,\pi), \\
        -2\pi f'(\pi), &\gamma =\pi,
    \end{cases}$$ is non-increasing on $(0, \pi]$.
    \label{lemma:upper-bound-decreasing}
\end{lemma}
\begin{proof}
    Let $\gamma \leq \alpha$. By convexity, using the straight line from $\gamma$ to $\pi$, we have for $t \in  [0,1]$, that
    $$f(t\gamma + (1-t)\pi) \leq tf(\gamma) + (1-t)f(\pi) = tf(\gamma).$$
    There exists a $t_0 \in [0,1]$ such that $t_0\gamma + (1-t_0)\pi = \alpha \Leftrightarrow t_0 = \frac{\pi - \alpha}{\pi - \gamma}$. Therefore we have
    $$f(\alpha) \leq \left(\frac{\pi - \alpha}{\pi - \gamma}\right)f(\gamma) \Longleftrightarrow \frac{f(\alpha)}{\pi - \alpha} \leq \frac{f(\gamma)}{\pi - \gamma} \Longleftrightarrow h(\alpha) \leq h(\gamma).$$
    Using that $f(\pi)=0$, we define $h$ at $\gamma = \pi$ by continuity.
\end{proof}

The idea to prove Theorem \ref{thrm:general-bounds} is to work inside the sum first by exploiting the convexity of $f$. We estimate above using a straight line segment from $\beta$ to $\pi$ and below by the tangent line at $\pi$, and then take the sum. 

\begin{proof}
    Let $\alpha \in [0,\pi]$. By convexity of $f$, using the tangent line of $f$ at $\pi$, and that $f(\pi)=0$ we have
    \begin{equation}\label{eqn:tangent-line-bound}
    f(\alpha) \geq f(\pi) + f'(\pi)(\alpha - \pi) = f'(\pi)(\alpha - \pi).
    \end{equation}
    Furthermore, equality holds if and only if $f$ is affine on $[\alpha, \pi]$.

    Let $\beta = \min A \in (0,\pi)$. Again by convexity of $f$, using the straight line segment from $\beta$ to $\pi$, and that $f(\pi) = 0$, we have for $t \in [0,1]$,
    $$f(t \beta + (1-t)\pi) \leq t f(\beta) + (1-t) f(\pi) = t f(\beta).$$
    We further have each $\gamma_i \in [\beta,\pi)$, that is the estimate holds for all $\gamma_i$, and hence we can write
    $$t \beta + (1-t) \pi = \gamma_i \Leftrightarrow t = \frac{\pi - \gamma_i}{\pi - \beta}.$$
    Therefore we have
    \begin{equation}
    f(\gamma_i) \leq \left(\frac{\pi - \gamma_i}{\pi - \beta}\right)f(\beta).
    \label{eqn:straight-line-bound}
    \end{equation}
    Now we notice
    $$\sum_{i=1}^N (\pi - \gamma_i) = N\pi - (N - 2) \pi = 2\pi.$$
    Hence, by \eqref{eqn:tangent-line-bound} and \eqref{eqn:straight-line-bound}, we have
    $$\sum_{i=1}^N f'(\pi)(\gamma_i - \pi) \leq \sum_{i=1}^N f(\gamma_i)\leq \sum_{i=1}^N \left(\frac{\pi - \gamma_i}{\pi - \beta}\right)f(\beta)$$
    which, after evaluating the sums, becomes
    $$ -2 \pi f'(\pi) \leq a_1(\Omega)\leq \frac{2 \pi f(\beta)}{\pi - \beta}.$$
\end{proof}

\begin{remark}
    We note that Theorem \ref{thrm:general-bounds} does not require that $f$ is decreasing. However, if $f$ were non-negative, then convexity would imply that $f$ is decreasing (and hence $-2 \pi f'(\pi) \geq 0$  as expected).
\end{remark}

Next we consider approximations of convex sets by convex polygons. Central to our proof is the fact that as we approximate the smooth domain by polygons, the interior angles of the polygons will tend to $\pi$ and so the corresponding straight line segments will approach the tangent lines.

\begin{corollary}[Convergence for polygons to smooth]\label{cor:general-conv}
    Suppose $(\Omega_k)_{k \geq 1}$ is a sequence of bounded, convex, open sets with polygonal boundary, $\Omega$ is a bounded, convex, open set with $C^\infty$ boundary and $\Omega_k \overset{H}{\longrightarrow} \Omega$. Then for $f \in \mathcal{C}$ we have
        $$\lim_{k \to \infty}a_1(\Omega_k) = -2\pi f'(\pi).$$
    In particular, if $a_1(\Omega) = -2 \pi f'(\pi)$, then $\lim_{k \to \infty} a_1(\Omega_k) = a_1(\Omega)$.
\end{corollary}
\begin{proof}
    Let $N_k$ be the number of interior angles of $\Omega_k$ and write $A_k = \{\gamma_{k,i}\}_{i=1}^{N_k}$ for the set of interior angles of $\Omega_k$. By \cite[Lemma 4.7]{MR24} we have that $\gamma_{k,i} \overset{k \to \infty}{\longrightarrow} \pi$ for all $i$. We therefore also have that $\beta_k = \min A_k \overset{k \to \infty}{\longrightarrow} \pi$.
    
    Using $f(\pi) = 0$, we rewrite the upper bound from Theorem \ref{thrm:general-bounds} as follows
    $$a_1(\Omega_k) \leq -2\pi \left(\frac{f(\beta_k)-f(\pi)}{\beta_k - \pi}\right) \overset{k \to \infty}{\longrightarrow} -2\pi f'(\pi).$$
    As the limit of the right-hand side agrees with the lower bound from Theorem \ref{thrm:general-bounds}, we conclude the limit as claimed.
\end{proof}

Now we consider approximations of convex polygons by convex sets. Note here we require the additional constraint that $f$ is strictly convex at $\pi$.

\begin{corollary}[Non-convergence for smooth to polygons]\label{cor:general-non-conv}
    Suppose $(\Omega_k)_{k \geq 1}$ is a sequence of bounded, convex, open sets with $C^\infty$ boundary, $\Omega$ is a bounded, convex, open set with polygonal boundary, and $\Omega_k \overset{H}{\longrightarrow} \Omega$. Let $f \in \mathcal{C}$ and further assume $f$ is strictly convex on a left neighbourhood of $\pi$. Then if $a_1(\Omega_k) = -2 \pi f'(\pi)$, we have $\lim_{k \to \infty} a_1(\Omega_k) < a_1(\Omega)$.
\end{corollary}
\begin{proof}
    Since $f$ is strictly convex at $\pi$, by Theorem \ref{thrm:general-bounds} we have that the lower bound is strict and hence taking limits we get $\lim_{k \to \infty} a_1(\Omega_k) = - 2\pi f'(\pi) < a_1(\Omega)$.
\end{proof}

We prove Theorem \ref{thm:conv} and \ref{thm:nonconv} for the Open heat content first.

\begin{proof}[Proof of Theorem \ref{thm:conv} and \ref{thm:nonconv} for the Open heat content]
    For this case we are working with the coefficients in \eqref{eq:Open-smooth-coefficients}, \eqref{eq:Open-poly-coeffs}, and \eqref{eq:Open-angle}. By Corollaries \ref{cor:general-conv} and \ref{cor:general-non-conv} it is enough to verify that $g \in \mathcal{C}$, $g$ is strictly convex at $\pi$, and that $-2 \pi g'(\pi) = 0$.  We compute that $g(\pi) = 0$ and $g''(\pi) = 2/3\pi > 0$. 
    So we have that $g \in \mathcal{C}$ and $g$ is strictly convex at $\pi$. Finally we have $g'(\pi) = 0$ by applying L'Hôpital's rule twice, hence the result follows.
\end{proof}

The idea to deal with the Dirichlet case is to push the limits inside the integral. We apply Theorem \ref{thrm:general-bounds} to the function under the integral. From here we consider the upper bound, to which we can apply dominated convergence by the observation of Lemma \ref{lemma:upper-bound-decreasing}. This allows us to move the limit inside the integral and take the limit on the level of functions instead.

\begin{proof}[Proof of Theorem \ref{thm:conv} and \ref{thm:nonconv} for the Dirichlet heat content]
    For the Dirichlet case we are considering the coefficients in \eqref{eq:Dir-smooth-coeffs}, \eqref{eq:Dir-poly-coeffs} and \eqref{eq:Dir-angle}.
    
    For $\alpha \in [0,\pi]$, we define
    $$g_\theta(\alpha) := \frac{\sinh((\pi - \alpha)\theta)}{\cosh(\alpha\theta)}$$
    and write
    $$b_0(\Omega_k) := \sum_{i=1}^{N_k} g_\theta(\gamma_{k,i}).$$
    The function $g_\theta$ is twice differentiable and taking derivatives yields
    \begin{align*}
    g_\theta'(\alpha) &= \frac{-\theta\cosh(\pi\theta)}{\cosh^2(\alpha\theta)}, \\
    g_\theta''(\alpha) &= \frac{2\theta^2\cosh(\pi\theta)\sinh{(\alpha \theta)}}{\cosh^3(\alpha\theta)} > 0 \quad \forall \theta \in (0, \infty), \forall\alpha \in (0,\pi].
    \end{align*}
    Also $g_\theta(\pi) = 0$ and 
    $$g_\theta'(\pi) = \frac{-\theta}{\cosh(\pi\theta)}.$$ 
    We deduce that $g_\theta \in \mathcal{C}$ and is strictly convex at $\alpha = \pi$. So, by Theorem \ref{thrm:general-bounds}, we have
    $$-2 \pi g_\theta'(\pi) < b_0(\Omega_k) \leq \frac{2\pi g_\theta(\beta_k)}{\pi - \beta_k},$$
    where $\beta_k = \min A_k$, and $\lim_{k \to \infty}b_0(\Omega_k) = -2 \pi g_\theta'(\pi)$ by Corollary \ref{cor:general-conv}.
    Now since
    \begin{align*}
    a_1(\Omega_k) &= \sum_{i=1}^{N_k} \int_0^\infty \frac{4}{\sinh(\pi\theta)} g_\theta(\gamma_{k,i})d\theta\\ &= \int_0^\infty \frac{4}{\sinh(\pi\theta)} \sum_{i=1}^{N_k} g_\theta(\gamma_{k,i})d\theta
    = \int_0^\infty \frac{4}{\sinh(\pi\theta)} b_0(\Omega_k) d\theta,
    \end{align*}
    we have
    $$\int_0^\infty \frac{8\pi\theta}{\sinh(\pi\theta)\cosh(\pi\theta)}d\theta < a_1(\Omega_k) \leq \int_0^\infty \frac{4}{\sinh(\pi\theta)} \frac{2\pi g_\theta(\beta_k)}{\pi - \beta_k} d\theta.$$
    We now wish to take the limit as $k \to \infty$. 
    By Lemma \ref{lemma:upper-bound-decreasing}, the function $\alpha \mapsto 2\pi g_\theta(\alpha)/(\pi - \alpha)$ is decreasing on $[0,\pi)$.
    In addition, we know as before that $\beta_k \overset{k \to \infty}{\longrightarrow} \pi$, and so there exists an $N \in \mathbb{N}$ such that $|\beta_k - \pi| < \pi/2 \Leftrightarrow \beta_k > \pi/2$ for $k > N$.
    Therefore, for $k > N$, we have 
    $$\frac{2 \pi g_\theta(\beta_k)}{\pi - \beta_k} \leq \frac{2\pi g_\theta(\pi/2)}{\pi - \pi/2} = 4 g_\theta(\pi/2).$$ 
    We also have that
    \begin{align*}
    \int_0^\infty \frac{4}{\sinh(\pi\theta)}(4g_\theta(\pi/2))d\theta &= 16\int_0^\infty \frac{1}{2\sinh(\pi\theta/2)\cosh(\pi\theta/2)}\frac{\sinh(\pi\theta/2)}{\cosh(\pi\theta/2)}d\theta \\
    &= 8 \int_0^\infty \frac{1}{\cosh^2(\pi\theta/2)}d\theta < \infty.
    \end{align*}
    For $k > N$, we thus conclude that 
    $$h_k\colon \theta \mapsto \frac{8\pi g_\theta(\beta_k)}{(\pi - \beta_k)\sinh(\pi \theta)}$$ 
    is a sequence of integrable functions dominated by the integrable function $\theta \mapsto 8 (\sech (\pi \theta/2))^2$, and hence by the dominated convergence theorem,
    \begin{align*}
    \lim_{k \to \infty} \int_0^\infty \frac{4}{\sinh(\pi\theta)} \frac{2\pi g_\theta(\beta_k)}{\pi - \beta_k} d\theta &= \int_0^\infty \frac{4}{\sinh(\pi\theta)} \lim_{k \to \infty}\left(\frac{2\pi g_\theta(\beta_k)}{\pi - \beta_k}\right) d\theta\\
    &= \int_0^\infty \frac{4}{\sinh(\pi\theta)}(-2\pi g_\theta'(\pi)) d\theta\\
    &= \int_0^\infty \frac{8\pi\theta}{\sinh(\pi\theta)\cosh(\pi\theta)}d\theta.
    \end{align*}
    Finally
    $$\int_0^\infty \frac{8\pi\theta}{\sinh(\pi\theta)\cosh(\pi\theta)}d\theta = \frac{4}{\pi}\int_0^\infty\frac{\theta}{\sinh(\theta)}d\theta = \pi,$$
    where the last equality follows from \cite[3.521.1]{GR14}. Hence we conclude that the limit is $\pi$ as required.

    On the other hand, for a polygon $P$ with $N$ interior angles denoted $\gamma_i$, we see that
    $$ a_{1,Q}(P) = \int_0^\infty \frac{4}{\sinh(\pi\theta)} \sum_{i=1}^{N} g_\theta(\gamma_i)d\theta > \int_0^\infty \frac{8\pi\theta}{\sinh(\pi\theta) \cosh(\pi \theta)} = \pi.$$
    Hence, for a sequence of smooth, convex domains $\Omega_k \overset{H}{\longrightarrow} P$, we have $\lim_{k\to \infty}a_{1,Q}(\Omega_k) = \pi < a_{1,Q}(P)$.
\end{proof}

\begin{remark}[Heat trace case]
    We can use this approach to recover the results \cite[Theorem 6.4]{MR25} and \cite[Theorem 6.5]{MR25} for the Dirichlet heat trace where the angular coefficient and smooth coefficient are given by \eqref{eq:Trace-angle} and \eqref{eq:Trace-smooth} respectively. We verify the conditions of Corollary \ref{cor:general-conv} and Corollary \ref{cor:general-non-conv}. Indeed $h$ is twice differentiable with
    $$h'(\alpha) = \frac{1}{24\pi}\left(-\frac{\pi^2}{\alpha^2} - 1\right),\quad
    h''(\alpha) = \frac{1}{24\pi}\left(\frac{2\pi^2}{\alpha^3}\right) > 0, \quad \text{for} \ \alpha > 0.$$
    So $h$ is strictly convex on $(0, \pi]$ and $h(\pi) = 0$. Finally $-2 \pi h'(\pi) = 1/6$ as required.
\end{remark}

\subsection*{AI Usage Declaration}
Generative AI (Codex, GPT-6.1 Sol) was used to illuminate a way to manipulate \eqref{eq:K} to obtain \eqref{eq:monotone} and reveal its monotone properties. The argument was checked carefully, simplified, and written in its final form by the authors.




\bibliographystyle{plain}              
\bibliography{References}

\end{document}